\documentclass[times,doublespace]{amsart}

\DeclareFontFamily{U}{matha}{\hyphenchar\font45}
\DeclareFontShape{U}{matha}{m}{n}{
      <5> <6> <7> <8> <9> <10> gen * matha
      <10.95> matha10 <12> <14.4> <17.28> <20.74> <24.88> matha12
      }{}
\DeclareSymbolFont{matha}{U}{matha}{m}{n}
\DeclareFontFamily{U}{mathx}{\hyphenchar\font45}
\DeclareFontShape{U}{mathx}{m}{n}{
      <5> <6> <7> <8> <9> <10>
      <10.95> <12> <14.4> <17.28> <20.74> <24.88>
      mathx10
      }{}
\DeclareSymbolFont{mathx}{U}{mathx}{m}{n}

\DeclareMathSymbol{\obot}         {2}{matha}{"6B}
\DeclareMathSymbol{\bigobot}       {1}{mathx}{"CB}

\usepackage[pdfauthor={Congling Qiu}, 
        pdftitle={???}%
      dvips,colorlinks=true]{hyperref}
\hypersetup{
  bookmarksnumbered=true,
  citecolor=black,
  pagecolor=black, 
  urlcolor=black,  
} 

\usepackage[utf8]{inputenc}

\usepackage{xtab}
\usepackage{amsmath, amssymb, cancel,verbatim}
\usepackage[all]{xy}
\usepackage[scr=euler,scrscaled=.96]{mathalfa}
 
\usepackage{enumerate}
\usepackage{mathtools,booktabs}
\usepackage{amscd}
\usepackage{bbm, stmaryrd}
\usepackage{yfonts}
\usepackage{color}
\usepackage{xcolor}
\usepackage[nameinlink]{cleveref}
\usepackage{epstopdf} 
\usepackage{booktabs}

\newtheorem{teo}{Theorem}[subsection]
\newtheorem{thm}[teo]{Theorem}
\newtheorem{thm*}[teo]{Theorem*}

\newtheorem{lem}[teo]{Lemma}
\newtheorem{cor}[teo]{Corollary}
\newtheorem{defn}[teo]{Definition}

\newtheorem{rmk}[teo]{Remark}

\newtheorem*{hypo*}{Hypothesis}

  \newtheorem{prop}[teo]{Proposition}
    \newtheorem {conj}[teo]{Conjecture}

\numberwithin{equation}{section}

    \newcommand{\BC}{{\mathbb {C}}} 
     \newcommand{\BF}{{\mathbb {F}}}

     \newcommand{\BP}{{\mathbb {P}}}
    \newcommand{\BQ}{{\mathbb {Q}}}

     \newcommand{\BZ}{{\mathbb {Z}}}

    \newcommand{\wt}{\widetilde}

    \newcommand{\pair}[1]{\langle {#1} \rangle}

 \newcommand{\lb}{\left(} \newcommand{\rb}{\right)}

    \newcommand{\Aut}{{\mathrm{Aut}}}

    \newcommand{\Ch}{{\mathrm{Ch}}}

    \newcommand{\cl}{{\mathrm{cl}}}

    \newcommand{\End}{{\mathrm{End}}}

     \newcommand{\GL}{{\mathrm{GL}}}

    \newcommand{\Hom}{{\mathrm{Hom}}}

    \newcommand{\Ind}{{\mathrm{Ind}}}
        \newcommand{\DC}{{\mathrm{DC}}}

    \newcommand{\PGL}{{\mathrm{PGL}}}

    \newcommand{\PSL}{{\mathrm{PSL}}}

    \renewcommand{\mod}{\mathrm{mod}\ }

    \newcommand{\sgn}{{\mathrm{sgn}}}

    \newcommand{\Tr}{{\mathrm{Tr}}}

\newcommand\supervisor[1]{\def\@supervisor{#1}}

\newcounter{elno}

 \newenvironment{altenumerate}
   {\begin{list}
      {(\theenumi) }
      {\usecounter{enumi}
       \setlength{\labelwidth}{0pt}
       \setlength{\labelsep}{0pt}
       \setlength{\leftmargin}{0pt}
       \setlength{\itemsep}{\the\smallskipamount}
       \renewcommand{\theenumi}{\roman{enumi}}
      }}
   {\end{list}}

\renewcommand{\cong}{\simeq}

 \author{Congling Qiu,\ Wei Zhang}

\begin{document} 
  \title{Vanishing results in Chow groups for the modified diagonal cycles}
\begin{abstract}

We prove a sufficient condition for the vanishing of the modified diagonal cycle in the Chow group (with $\BQ$-coefficients) of the triple product of a curve over $\BC$. 
 We exhibit infinitely many non-hyperelliptic curves, including the Fricke--Macbeath curve, the Bring curve, and two 1-dimensional families parameterized by certain Hurwitz spaces, for which our condition is satisfied.
 \end{abstract}
 \maketitle 
 \tableofcontents 
  
 \section{Introduction}

 \subsection{The modified diagonal cycle in the Chow group}\label{vanish}
  For a smooth projective  variety $X$ over a  field $F$, let $\Ch^i(X)$ be the Chow group  of $X$ (the group of  codimension $i$ cycles with   $\BQ$-coefficients, modulo rational equivalence).  The diagonal cycle on $X\times X $ plays an important role in many aspects. For the multi-folded product $X^n=X\times\cdots\times X$ ($n$-times), the diagonal cycle $\Delta$ (sometimes called the small diagonal when $n\geq 3$) has also been studied by many authors. Of particular interest is the modified diagonal cycle in the case $n=3$ and $\dim X=1$, also called the Gross--Schoen cycle \cite{GS,Zha10}. This will be the main focus of this paper, and we refer to \cite{BV,OG,Vo} for some other classes of $X$ and more general $n$. For example, by \cite[Proposition 2.4]{OG},  the
vanishing of the  modified diagonal cycle in the case $n=3$ implies the vanishing for $n\geq 3$.
  
  Throughout this paper, a ``curve" always means a ``smooth projective curve" over a field $F$. Let $X$ be a curve over $F$. For any $e\in \Ch^1(X)$ with $\deg e=1$,   there is the modified diagonal cycle   $[\Delta_{e}]\in \Ch^2(X^3)$ defined in \cite{GS}. 
  If $e=[p]$ the class of an $F$-point $p$, we set 
\begin{alignat*}{3}\Delta_{12}&= \{(x,x,p):x\in X\},\quad \Delta_{23}&&= \{(p,x,x):x\in X\}, \quad
\Delta_{31}& &= \{(x,p,x):x\in X\}, \\
\Delta_{1}&=  \{(x,p,p):x\in X\},\quad \ 
\Delta_{2}&&=   \{(p,x,p):x\in X\},\quad  \
\Delta_{3}& &= \{(p,p,x):x\in X\}.
\end{alignat*}
Then $[\Delta_{e}]$ is  the   class  in $\Ch^2(X^3)$ of the algebraic cycle
\begin{align}\label{eq:mod diag}
\Delta-\Delta_{12}-\Delta_{23}-\Delta_{31}+
\Delta_{1}+\Delta_{2}+
\Delta_{3}.
\end{align}
 In general,  we can define $[\Delta_{e}]$ using \eqref{GScycle} (a formula similar  to \eqref{eq:mod diag} for $[\Delta_e]$   is given 
   in  \cite[(1.1)]{Zha10}).  

The modified diagonal $[\Delta_{e}]$ has trivial cohomology class, by Proposition \ref {2dneq n'}. 
When $X$ is a hyperelliptic curve, Gross and Schoen showed that  $[\Delta_e]$ vanishes in $\Ch^2(X^3)$ if $e=[p]$ for any Weierstrass point $p$. 
Below, we always assume  that a curve has genus at least $2$. By Proposition \ref{prop:nonvan},  the class  $[\Delta_e]$ does not vanish unless $e$ is equal to
\begin{align}\label{eq:xi}
\xi=\xi_X:=\frac{1}{\deg c_1(X)} c_1(X)\in \Ch^1(X).
\end{align}

The modified diagonal cycle is closely related to the Ceresa cycle \cite{C}, i.e., $[X]-(-1)^\ast[X]\in \Ch^{g-1}(J_X)$, where $J_X$
 is the Jacobian variety of $X$ and we embed $X$ into $J_X$ by $x\mapsto x-\xi$. The vanishing of  $[\Delta_{\xi}]\in \Ch^2(X^3)$ is equivalent to the vanishing of the Ceresa cycle in $ \Ch^{g-1}(J_X)$ (see  \cite[Theorem 1.5.5]{Zha10}). Ceresa \cite{C} proved that the class $[X]-(-1)^\ast[X]\in \Ch^{g-1}(J_X)$, hence $[\Delta_{\xi}]\in \Ch^2(X^3)$, does not vanish  for a {\em general} curve $X$ over $\BC$.  We recall another related ``non-vanishing result": by the ``Northcott property" of S. Zhang \cite[Theorem 1.3.5]{Zha10}, it is more often for the class $[\Delta_{\xi}]\in \Ch^2(X^3)$ to be non-vanishing in a proper (non-isotrivial) family of curves defined over a number field.

However,  it was not known whether  there exists any non-hyperelliptic curve $X$ such that $[\Delta_{\xi}]\in \Ch^2(X^3)$ vanishes. 
The analogous question with the $\ell$-adic intermediate Jacobian, the complex  intermediate Jacobian, and  the Griffiths group (Chow group modulo algebraic equivalence) respectively in place of the Chow group,  has been recently answered affirmatively by Bisogno--Li--Litt--Srinivasan \cite{BLLS},  Lilienfeldt \cite{Lil}, and Beauville--Schoen \cite{BS} respectively, using different curves. In this paper we will provide   many examples, including  two 1-dimensional families, of non-hyperelliptic $X$ over $\BC$ with  $[\Delta_{\xi}]=0\in \Ch^2(X^3)$. In fact, all  of them are defined over number fields (some are even over $\BQ$). 
Consequently, we also obtain vanishing results for the Ceresa cycle $[X]-(-1)^*[X]$ in the Chow group of the Jacobian variety $J_X$.
 
 Before stating our main result, let us point out a curious consequence of the vanishing of $[\Delta_{\xi}]\in \Ch^2(X^3)$. Let $X$ be a curve defined over a number field.   By \cite[Corollary 1.3.2]{Zha10}, the vanishing of $[\Delta_\xi]$ in the Chow group gives a new proof of the positivity of the self-intersection number of the admissible dualizing sheaf: $\omega_a^2>0$ (see the notation in  {\it loc. cit.}), which was a deep theorem due to Ullmo \cite{U} and S. Zhang \cite{Z-98} independently. In particular,  by S. Zhang's theorem of successive minima (see \cite[\S2]{Z-ICM}), this positivity in turn implies the Bogomolov conjecture for $X$.

\subsection{A vanishing theorem}

Our result relies on  the following sufficient condition for the vanishing of $[\Delta_\xi]$.  From now on we work with the ground field $F=\BC$, except in \S\ref{sec:Finite fields}. Let $H^1(X)$ denote the first Betti cohomology $H^1(X,\BC)$.  Let $G$ be a finite group acting on $X$ by automorphisms. Then there is the induced action of $G$ on $H^1(X)$. 
On $H^1(X)^{\otimes 3}$,   
there is  the induced diagonal $G$-action.
    We have the following theorem, proved in \ref{ss:proof}. 
     \begin{thm} \label{th:van} If $(H^1(X)^{\otimes 3})^{G}=0$, then  the class $[\Delta_\xi]$ vanishes in $\Ch^2(X^3)$.
      \end{thm}    
 
 Some remarks are in order.
 A necessary condition for $(H^1(X)^{\otimes 3})^{G}=0$ is $H^1(X)^{G}=0$. Hence the quotient  curve $X/G$ has genus zero.

A special case of Theorem \ref{th:van}  is when $X$ is hyperelliptic and $G=\BZ/2\BZ$ is generated by an involution that gives rise to a degree two map $X\to \BP^1$. Then $H^1(X)^G=0$ and hence $G$ acts on $H^1(X)$ by the unique nontrivial character. It follows that $(H^1(X)^{\otimes 3})^{G}=0$. 
This gives a new proof of the aforementioned theorem of Gross and Schoen.

If $X$ is defined over a number field,  Theorem \ref{th:van} is also predicted by the conjectural   injectivity of the Abel--Jacobi map (due to Beilinson and Bloch).

To the best of our knowledge, the only other known result  (besides the trivial ones for varieties such as $\BP^2$) on the vanishing (in Chow group) of the modified diagonal cycle on the triple product of a variety is 
due to Beauville and Voisin \cite{BV}, who showed that when $X$ is a K3 surface, $e=\frac{1}{\deg c_2(X)}c_2(X)\in \Ch^2(X)$ (note that $\deg c_2(X)=24$) is the class of  an $F$-point, and  the class  of  $[\Delta_e]$ similarly defined  by \eqref{eq:mod diag}  vanishes in $\Ch^4(X^3)$.

\subsection{Hurwitz curves}\label{Fricke--Macbeath curve}

  A curve over $\BC$ of genus $g$ is called a Hurwitz curve if its automorphism group achieves the maximal possible order $84(g-1)$. 
There is a unique (up to isomorphism) Hurwitz curve $X$ of genus 7 over $\BC$, known as the Fricke--Macbeath curve \cite{Mac1}. Its automorphism group is isomorphic to  $ \PGL_2({\BF_8})$.
 Shimura  \cite{Shi} identified   it  as a Shimura curve.

     By a theorem of   Bisogno, Li, Litt,  Srinivasan    \cite{BLLS} (see also 
        Gross \cite{Gro}), we have $(H^1(X)^ {\otimes 3})^{\Aut(X)}=0$ for the  Fricke--Macbeath curve $X$.         We thus obtain
     \begin{thm} \label{th:FM} Let $X$ be the Fricke--Macbeath curve.
Then the class $[\Delta_\xi]$ vanishes in $\Ch^2(X^3)$. 

      \end{thm}    
       We will also provide an independent proof of 
       $(H^1(X)^ {\otimes 3})^{\Aut(X)}=0$ for the  Fricke--Macbeath curve $X$
        in \S\ref{Fricke--Macbeath curve1}. This will also help us find another example of similar sort, the Bring curve in \S\ref{ss:Bring}, which has the automorphism group $\PGL_2(\BF_5)$, the largest possible automorphism group for genus $4$ curves.

 We make the following conjecture. 
\begin{conj}\label{conj H}
The cycle $[\Delta_\xi]$ vanishes in  $\Ch^2(X^3)$ for only finitely many Hurwitz curves $X$.
\end{conj} 
Note that Hurwitz curves are never hyperelliptic, cf. Lemma \ref{cor:Hur}. The authors would not be surprised if the Fricke--Macbeath curve turns out to be the only Hurwitz curve with vanishing $[\Delta_\xi]$ in the Chow group.

As an evidence, we prove that, among those Hurwitz curves with $G=\Aut(X)\simeq \PSL_2(\BF_q)$,  the Fricke--Macbeath curve is the only one satisfying the condition $(H^1(X)^ {\otimes 3})^{G}=0$, see Theorem \ref{thm:Hur PSL}.

\subsection{More examples}
In addition to the Fricke--Macbeath curve, we have also found other examples satisfying  $(H^1(X)^ {\otimes 3})^{G}=0$ for every genus $g\in \{3,4,5\}$. We have also found an 1-dimensional family in genus $4$ and $5$ respectively. See \S\ref{s:ex}. We use the data obtained in \cite{MSSV}.

 Based on our computation, we would like to make the following two conjectures.  Let $\mathcal{M}_g$ be the moduli stack  of genus $g$ curves (over $\BC$).

 \begin{conj}\label{conj g}
  \begin{altenumerate}

 \item[(a)] For $g> 2$, there exists a genus $g$ non-hyperelliptic curve $X$ with  $[\Delta_\xi]=0$ in $\Ch^2(X^3)$.
 \item[(b)] For any integer $d$, there exist an integer $g$ and a connected smooth family of  genus $g$ curves $f:\mathcal{X}\to \mathcal{B}$ such that under the map $\mathcal{B}\to \mathcal{M}_g $ induced by $f$, the image of $\mathcal{B}$ has dimension at least $d$, and such that, for every $b\in \mathcal{B}(\BC)$,  the fiber $X_b$ of $f$ over $b$ is non-hyperelliptic and has $[\Delta_\xi]=0$ in $\Ch^2(X_b^3)$.
 \item[(c)] The image of $\mathcal{B}$ in $ \mathcal{M}_g $ in part (b) has dimension at most $2g-1$ (the dimension of the hyperelliptic loci in  $\mathcal{M}_g$). 

 \end{altenumerate}

 \end{conj}

\begin{conj}\label{conj g'}
\begin{altenumerate}
\item[(a)] There exist   non-hyperelliptic curves $X$ of arbitrary large genera with  $$(H^1(X)^{\otimes 3})^{\Aut(X)}=0.$$
\item[(b)] For any integer $d$, there exist an integer $g$ and a connected smooth family of  genus $g$ curves $f:\mathcal{X}\to \mathcal{B}$ such that under the map $\mathcal{B}\to \mathcal{M}_g $ induced by $f$, the image of $\mathcal{B}$ has dimension at least $d$, and such that, for every $b\in \mathcal{B}(\BC)$,  the fiber $X_b$ of $f$ over $b$ is non-hyperelliptic and has $(H^1(X_b)^{\otimes 3})^{\Aut(X_b)}=0$.

\end{altenumerate}

\end{conj} 

 For the first open case in Conjecture \ref{conj g} (a), we have not found any example of genus $g=6$ curve with vanishing $[\Delta_\xi]$. 
In fact,    we  expect that  all non-hyperelliptic  genus  6 curves $X$ have $(H^1(X)^{\otimes 3})^{\Aut(X)}\neq 0$ (which also explains the    formulation of Conjecture \ref{conj g'} (a)). 
This is indeed the case for
 non-hyperelliptic  genus  6 curves   in \cite{MSSV}.
  And those $X$ not in \cite{MSSV}  by definition have ``small" automorphism groups so that  $(H^1(X)^{\otimes 3})^{\Aut(X)}$ are less likely to be $ 0$.

  As suggested by part (a)'s of  these two conjectures,  $(H^1(X)^{\otimes 3})^{\Aut(X)}=0$ should  not be a necessary condition for the vanishing of $[\Delta_\xi]$. This can already be seen in genus 3, for example,
   from  the genus 3 quotient of  the Fricke--Macbeath curve  (Corollary \ref{cor:qt FM} and Lemma \ref{lem:qt FM}). As another example, if we assume a  conjecture of Beilinson and Bloch relating  Chow groups and $L$-functions, then the genus 3 curve used by Beauville  and Schoen \cite{BS} should have vanishing $[\Delta_\xi]$ in the Chow group, even though it does not satisfy $(H^1(X)^{\otimes 3})^{\Aut(X)}=0$. 
In our next paper \cite{QZ2}, we will provide more examples using Shimura curves.

        \subsection{Chow--K\"unneth modified diagonal cycle and Faber--Pandharipande cycle}
        
  Our proof of Theorem \ref{th:van} on the vanishing of the Gross--Schoen modified diagonal cycle is  based on the study of the 
 Chow--K\"unneth modified diagonal cycle on $X^3$ \eqref{CKcycle}. It 
coincides with $[\Delta_e]$ when $e$  is the class of  an $F$-point, and in general differs from $[\Delta_e]$ 
by pullbacks of the following Faber--Pandharipande cycle  
(see  \eqref{diffcycle}).

        Let  $\delta:X\to X^2$ be the   diagonal embedding.        Following a construction of 
             Faber and Pandharipande, define a zero-cycle    (see for example \cite{GG})
             \begin{equation}\label{FPcycle}  
        z_e  :=e\times e-\delta _*e \in \Ch^2(X^2).
       \end{equation}
For $e=\xi$, the class $z_\xi$ is known to vanish if $X$ is hyperelliptic. Faber and Pandharipande proved $z_\xi=0$ if the genus of $X$ is $3$; Green and Griffiths \cite{GG} proved that $z_\xi\neq 0$ if $X$ is a general curve of genus $g\geq 4$.
 
When $e=\xi$, the vanishing of either modified diagonal cycle implies the vanishing of the other modified diagonal cycle  as well as 
  $z_\xi$, by Proposition \ref{prop:nonvan}.
   We prove   the vanishing of  the Chow--K\"unneth modified diagonal cycle assuming $(H^1(X)^{\otimes 3})^{G}=0$ (a special case
   of Theorem \ref{good}).  Then we have Theorem \ref{th:van} and    the following theorem, proved in \S\ref{ss:proof}.
           \begin{thm} \label{th:vanFP} If $(H^1(X)^{\otimes 3})^{G}=0$, then  the class $      z_ \xi  $ vanishes in $\Ch^2(X^2)$.
      \end{thm}  
    
     \subsection{Finite fields}\label{sec:Finite fields}
   Finally in this introduction, we let $F$ be a finite field. 
By a theorem of Soul\'e \cite[Theorem 3]{Sou},  on a product of curves  
 over   $F$,  Chow 1-cycles  with $\BQ_\ell$-coefficients coincide (via the cycle class map) with 
 Tate 1-cycles with $\BQ_\ell$-coefficients. Here, $\ell$ is  different from  the characteristic of $F$.
 In particular, since  modified diagonal cycles are cohomologically trivial, they vanish in the Chow groups. In the proof of Soul\'e's theorem, the Frobenius map plays a similar role to the automorphism group in our result.   
 
   \subsection*{Acknowledgment}
The work started from an attempt to answer a question
of Dick Gross in his talk at the Kudla Fest in 2021 where he asked whether one can show the vanishing of the modified diagonal cycle in the Chow group for the Fricke--Macbeath curve. We thank Frank~Calegari, Dick~Gross, Daniel~Litt, Jennifer~Paulhus, Bjorn~Poonen, Dipendra~Prasad, Junliang~Shen, Burt~Totaro, 
Shouwu~Zhang for their comments and helpful discussions.

 C. Q.  is partially supported by the NSF grant DMS \#2000533.
 W. Z. is partially supported by the NSF grant DMS \#1901642.

\section{Proof of the vanishing theorem}
   \label{Product of curves}
 
 In this section we prove Theorem \ref{th:van}. The key observation is that
one can explicitly  describe  self-correspondences  on a  curve annihilating the cohomology of a particular degree.

\subsection{Divisorial correspondences}\label{Divisorial correspondences}
  Let $X$ be a curve over $\BC$.  Consider the cycle class map,  
 $$\cl:\Ch^{1}(X^2)\to H^2(X^2)\cong \bigoplus_{i=0}^2 \End(H^i(X)).$$
 Here the isomorphism is defined by K\"unneth decomposition  and Poincar\'e duality. 
  
 Fix $e\in \Ch^1(X)$    of degree 1. 
 \begin{defn}
  Let $\DC(X^2)$ be the subspace of  $\Ch^{1}(X^2)$ of $z$'s such that $z_* e\in \BQ e$ and $z^*e\in \BQ e$.  
  Let $\DC^1(X^2)$ be the subspace of  $\Ch^{1}(X^2)$ of $z$'s such that $z_* e=0$ and $z^*e=0$.  
  \end{defn}
  We call  cycles  in $\DC^1(X^2)$  divisorial correspondences with respect to  $e$.
   For notational convenience,  we also define  $$ \DC^0(X^2)= \BQ(e\times [X]),\quad  \DC^2(X^2)= \BQ([X]\times e).$$
For $z\in \DC(X^2)$, we assign 
$$
 \lb    z^*e \times [X],  z-z^*e  \times [X]-[X]\times  z_* e,
[X]\times   z_* e\rb
$$ 
to get the following direct sum decomposition (the sum is clearly direct):  
   \begin{equation}\label{refm2}
    \DC(X^2)=\DC^0(X^2)\oplus  \DC^1(X^2)\oplus \DC^2(X^2) .
    \end{equation}
    
    The following lemma should be well-known. For the sake of completeness we include a proof.
     \begin{lem}\label{clinj}

    The restriction $\cl|_{\DC(X^2)}$ is injective.  
 
    \end{lem} 

\begin{proof}Note that the image of  $\cl|_{\DC^i(X^2)}$ is in $\End(H^i(X))$ for $0\leq i\leq 2$. Therefore it suffices to show the injectivity of $\cl|_{\DC^i(X^2)}$ for every $i\in\{0,1,2\}$. The maps $\cl|_{\DC^0(X^2)},\cl|_{\DC^0(X^2)}$ are obviously injective. Note that there is a natural isomorphism (for a proof, see  \cite[Theorem 3.9]{Sch})
$$
\DC^1(X^2)\simeq \End^0(J_X),
$$where $J_X$ is  the Jacobian variety of $X$, and $\End^0(J_X)$ denotes the $\BQ$-endomorphism ring of the abelian variety $J_X$.
Since the natural map $\End^0(J_X)\to \End(H^1(X))$ is injective, the map $\cl|_{\DC^1(X^2)}$ is injective and the proof is complete.      \end{proof}
   
  Let $\delta:X\to X^2$ be the   diagonal embedding.
 \begin{rmk}
We may view $\Ch^1(X^2)$ as the ring of self-correspondence on the curve $X$. Then $\DC(X^2)$ is a subring and  the identity   is the diagonal  cycle  $[\delta_* X]$. We have a decomposition (as $\BQ$-vector spaces)
$$
\Ch^1(X^2)= \DC(X^2)\oplus (\pi_1^\ast \Ch^1(X)_0\oplus \pi_2^\ast \Ch^1(X)_0),
$$
where $\pi_1,\pi_2: X^2\to X$ are the two projection maps, and  $\pi_1^\ast \Ch^1(X)_0\oplus \pi_2^\ast \Ch^1(X)_0$ is an ideal of the ring $\Ch^1(X^2)$.
  \end{rmk}

 For a fixed $e\in\Ch_0(X)$,  let  $\delta_{e}$ be the projection of $[\delta_* X]$  to $\DC^1(X^2)$ via \eqref{refm2}, i.e.,
   \begin{equation}
  \label{delta}
\delta_{e}:=[\delta_* X]-(e\times [X] +[X]\times e)\in  \DC^1(X^2).
 \end{equation}   It is straightforward to check the following lemma.
 \begin{lem} \label{zdelta}The left/right composition by $\delta$ on  $\DC(X^2)$  is the projection to $\DC^1(X^2)$ via \eqref{refm2}.
 \end{lem}

   \subsection{Product of curves}\label{Application to a product of curves} 
Let $X^n$  be the $n$-th fold self-product of $X$. We idenitify $X^n\times X^n=X^{2n}$, sending $((x_1,\cdots, x_n),(y_1,\cdots, y_n))\in X^n\times X^n$ to $(x_1,\cdots, x_n,y_1,\cdots, y_n)\in X^{2n}$. Let $\pi_{i,j}:X^{2n}\to X\times X$ be the projection map to the $(i,j)$-th factor.
 We have a natural map
 $$
 \xymatrix{ \bigotimes_{i=1}^n\Ch^1(X^2)\ar[r]& \Ch^n(X^n\times X^n)}
 $$which sends a tensor $z_1\otimes \dots \otimes z_n$ to the intersection product of divisors $\pi_{1,n+1}^\ast (z_1), \dots,  \pi_{n,2n}^\ast (z_n)$ on $X^{2n}$.
 
 Restricting this map to its subspace   and applying the cycle class maps, we have a commutative diagram:
 \begin{equation}\label{eq:prod}
\xymatrix{ \bigotimes_{i=1}^n\DC^1(X^2)\ar[r]\ar[d]^{\cl}& \Ch^n(X^n\times X^n)\ar[d]^{\cl}\\
\bigotimes_{i=1}^n \End(H^{1}(X))\ar[r]
& \End( H^n(X^{n}))}
 \end{equation}
 where  the bottom map is induced by  the K\"unneth decomposition. By Lemma \ref{clinj}, the left map is injective.  By the obvious injectivity of the bottom map, the top map of \eqref{eq:prod} is also injective.  We will therefore view an element in $\bigotimes_{i=1}^n\DC^1(X^2)$ as an element in $ \Ch^n(X^n\times X^n)$.
By 
  the K\"unneth decomposition, we have the following proposition.
 \begin{prop}\label{2dneq n}
 If $m\neq n$,  $\bigotimes_{i=1}^n\DC^1(X^2)$ annihilates $H^m(X^n)$, by either pushforward or pullback.
 \end{prop}

  For a class $z\in\DC(X^2) $, we denote
 \begin{equation}
 z^{\otimes n}: =\bigotimes_{i=1}^n z \in \bigotimes_{i=1}^n \DC(X^2).
 \end{equation} 
We will use the Chow--K\"unneth  projector  $\delta_e^{\otimes n}\in  \bigotimes_{i=1}^n \DC^1(X^2)$ for $\delta_e$ defined by \eqref{delta}. We remind the reader that this is \textit{not}  the  projector used by Gross and Schoen  \cite[Section 2]{GS} (in the case that $e $ is the class of an $F$-point), as we now discuss.

  \subsection{Modified diagonals}
 We will be  mainly interested in the case $n=3$.   The Gross--Schoen modified diagonal cycle with respect to $e$ is 
 \begin{equation}
 \label{GScycle}
[\Delta_e]= ([\delta_* X]- [X]\times e)^{\otimes 3}_*[\Delta]\in \Ch^2(X^3).
 \end{equation}  
   (It is straightforward to check that $[\Delta_e]$ coincides with the one in \cite[(1.1)]{Zha10} generalizing  \eqref{eq:mod diag}.)
Similar to  Proposition \ref {2dneq n}, 
we have the following proposition.
 \begin{prop}[{\cite[Proposition 3.1]{GS}}]\label{2dneq n'}
The cohomology class of  $[\Delta_{e}]$ is 0. 
 \end{prop}

Define the Chow--K\"unneth  modified diagonal cycle to be
  \begin{equation}
 \label{CKcycle}
 \lb\delta_e^{\otimes 3}\rb_*[\Delta]\in \Ch^2(X^3),
 \end{equation}  
 whose cohomology class is 0 by  Proposition \ref {2dneq n}.
Let $\pi_{ij}:X^3\to X^2$, $i< j\in\{1,2,3\}$,  be the projection to the product of $i$-th and $j$-th $X$.
It is straightforward to check  
\begin{equation}
 \label{diffcycle}
\lb\delta_e^{\otimes 3}\rb_*[\Delta]-[\Delta_e] =\sum_{i< j }\pi_{ij}^*z_e.
 \end{equation}  
 Here  $z_e$ is the Faber--Pandharipande cycle \eqref{FPcycle}. Two immediate consequences of \eqref{diffcycle} are as follows. 
 First,  since $z_e$ has cohomology class 0,
 Proposition \ref {2dneq n} implies Proposition \ref {2dneq n'}.
Second,  if $e$  is the class of  an $F$-point,  then   $z_e=0$ and  thus
$
\lb\delta_e^{\otimes 3}\rb_*[\Delta]=[\Delta_e] 
$. In general, we have the following proposition.

  \begin{prop}\label{prop:nonvan} 
  Let   the  genus $g$ of $X$ be at least $2$.
  
  (1) If
 $[\Delta_e]= 0$,  then $e$ is equal to
 $\xi $ defined by \eqref{eq:xi} and $z_\xi=0$. 

(2) If $\lb\delta_\xi^{\otimes 3}\rb_*[\Delta]=0$, 
  then $z_\xi=0$.

 (3) $\lb\delta_\xi^{\otimes 3}\rb_*[\Delta]=0$ if and only if  $[\Delta_\xi]= 0$.
  \end{prop}
  
  \begin{proof}
(1) Consider  the intersection product $[\Delta_e]\cdot ( [X]\times [\delta_*X])$.
  By the adjunction formula, we have $[\delta_*X] \cdot [\delta_*X]=\delta_*  c_1(X)$. 
  In particular,  
  $$[\Delta]\cdot ([X]\times [\delta_*X])= ( [\delta_*X]\times  [X]) \cdot  ([X]\times  [\delta_*X] ) \cdot ([X]\times  [\delta_*X] )=\Delta_*c_1(X). $$
 Here  for $\Delta_*c_1(X)$, we abuse notation and understand $\Delta$ as the diagonal embedding of $X$ in $X^3$. 
   Then  a direct computation shows that 
 \begin{equation}
 \label{eq:nonvan}
 [\Delta_e]\cdot ( [X]\times [\delta_*X])=\Delta_*c_1(X)-e\times \delta_*c_1(X)-2 \Delta_*e+2e\times \delta_*e. 
  \end{equation}  
   Its projection to the first $X$ is $\lb \deg c_1(X) \rb (\xi-e)$.  Since $\deg c_1(X)=2-2g\neq 0$,
   $[\Delta_e]= 0$ implies $e=\xi $.
   Let $e=\xi $.
 Then the projection of \eqref{eq:nonvan} to the product $X^2$ of the first two $X$'s is 
 $2g z_\xi$.   So  $[\Delta_e]= 0$  implies $z_\xi=0$.
 
 (2) By \eqref{diffcycle} and \eqref{eq:nonvan}, it is straightforward to show that
  the projection of  $\lb\delta_\xi^{\otimes 3}\rb_*[\Delta]\cdot ( [X]\times [\delta_*X])$
 is $(2g+2)z_\xi$. So  $\lb\delta_\xi^{\otimes 3}\rb_*[\Delta]=0$ implies  $z_\xi=0$. 
 
  (3) This now follows from \eqref{diffcycle}, the assertions in (1) and (2).
  \end{proof}

     \begin{prop}\label{prop:pf}
     Let $\pi:X\to Y$ be  a non-constant morphism of curves of genus at least $2$. If $[\Delta_{\xi_X}]=0 \in \Ch^2(X^3)$, then
 we have  $\pi_\ast(\xi_X)=\xi_Y\in \Ch^1(Y)$ and $[\Delta_{\xi_Y}]=0 \in \Ch^2(Y^3)$.
  \end{prop}
  
  \begin{proof}We note that $(\pi^3)_\ast([\Delta_{\xi_X}])= \deg(\pi)[\Delta_{\pi_\ast(\xi_X)}] \in \Ch^2(Y^3)$. Under the assumption $[\Delta_{\xi_X}]=0$, we have 
  $[\Delta_{\pi_\ast(\xi_X)}]=0$. This in turn implies that $\pi_\ast(\xi_X)=\xi_Y$  by Proposition \ref{prop:nonvan} (1)  and $[\Delta_{\xi_Y}]=0 $.
\end{proof}
\begin{rmk}It follows that a necessary condition for $[\Delta_{\xi_X}]=0 \in \Ch^2(X^3)$ is that
 $\pi_\ast(\xi_X)=\xi_Y\in \Ch^1(Y)$ holds for every  non-constant morphism $\pi: X\to Y$ with $g(Y)\geq 2$. We are not aware of a more direct characterization of curves $X$ with the latter property.

\end{rmk}

         \subsection{Proof of the vanishing theorem}\label{ss:proof}
Now let $G$ be a finite group acting on $X$ by automorphisms. Let $\xi\in \Ch^1(X)$ be a $G$-invariant divisor class of degree 1. In fact, such a $\xi$ clearly exists by averaging any degree 1 divisor class, and is in fact unique if $H^1(X)^G=0$.  We will simply take $\xi=\xi_X$ defined by \eqref{eq:xi} when the genus of $X$ is not $1$.
  For $g\in G$,  let $\Gamma_g$ be the  graph of  the automorphism of $X$ given by  $g$.
The $G$-invariance of $\xi$ implies that $[\Gamma_g]\in \DC(X^2)$, and 
   satisfies
 \begin{equation}
 \label{GgGg}
 [\Gamma_g]\circ\delta_\xi=\delta_\xi\circ [\Gamma_g]
 \end{equation}
 for the correspondence compositions. 
 In particular,   $(\delta_\xi^{\otimes 3})_{*}\Ch^*(X^3)$ 
   is stable under the {$G^3$-action.

 \begin{thm}\label{good} If $(H^1(X)^{\otimes 3})^{G}=0$ under the diagonal $G$-action, then there is  no nonzero diagonal-$G$-invariant  element in $(\delta_\xi^{\otimes 3})_{*}\Ch^*(X^3)$, or equivalently, $\delta_\xi^{\otimes 3}$ annihilates the diagonal-$G$-invariant  elements in $ \Ch^*(X^3)$. 
     \end{thm}
      \begin{proof} 
Consider the following element, by  averaging the diagonal $G$-action:
  $$z :=\frac{1}{|G|}\sum_{g\in G}(\Gamma_g\circ \delta_\xi)^{\otimes 3}\in  \DC(X^2) ^{\otimes 3}.$$
 By Lemma \ref{zdelta}, we have $z\in  \DC^1(X^2) ^{\otimes 3}$.   By the assumption $(H^1(X)^{\otimes 3})^{G}=0$, 
  $z$ acts on  $H^1(X)^{\otimes 3}$ by the zero map. Therefore the cycle class of $z$ vanishes under the left map of \eqref{eq:prod}. It follows from the   injectivity  of this map that $z=0\in  \DC^1(X^2) ^{\otimes 3}$.  
  
 Since the diagonal-$G$-invariant  elements in $(\delta_\xi^{\otimes 3})_{*}\Ch^*(X^3)$ are exactly $ z_\ast \Ch^*(X^3)$, the first assertion follows.   The equivalence of the two assertions follows from \eqref{GgGg} and the proof is complete.
  \end{proof}

     \begin{rmk}  
The proof of  Theorem \ref{good} can be easily generalized to show that there is a decomposition of representations of $G^3$
$$
 \Ch^*(X^3)_\BC= \bigoplus_{\Pi}  \Ch^*(X^3)_\BC[\Pi]
$$
where $\Pi$ runs over all irreducible $\BC$-representations of $G^3$ that appear in $(H^1(X))^{\otimes 3}$, and $ \Ch^*(X^3)_\BC[\Pi]$ is the $\Pi$-isotypic subspace of  $\Ch^*(X^3)_\BC$.
This can then be applied to study the subspace of $ \Ch^*(X^3)$ generated by $[\Delta_\xi]$ under the $G^3$-action (namely, the twisted diagonal cycles). Moreover, it  is also easy to   generalize  the above discussion to an arbitrary  $n$, instead of $n=3$.    \end{rmk}  
   
Now we can prove Theorem \ref{th:van} and  Theorem \ref{th:vanFP}.  Clearly the class $[\Delta]\in \Ch^2(X^3)$ is diagonal-$G$-invariant. Assume  $(H^1(X)^{\otimes 3})^{G}=0$. Then by
 Theorem \ref{good},   $(\delta_\xi^{\otimes 3})_\ast[\Delta]=0$. Hence   by Proposition \ref{prop:nonvan} (2), we obtain $z_\xi=0$ and Theorem \ref{th:vanFP} is proved. By Proposition \ref{prop:nonvan} (3), we obtain $[\Delta_\xi]=0$ and Theorem \ref{th:van} is proved.

    \section{Hurwitz curves} 
   
     \subsection{Hurwitz representation}\label{ss:Hur}
    We first recall the Hurwitz space \cite[\S2, \S3]{MSSV}. Fix a finite group $G$ and an ordered tuple $\mathcal{C}=(C_1,\cdots, C_r)$ of non-trivial conjugacy classes $C_i$ of $G$. A $G$-curve is a pair $(X,\iota)$ where
$X$ is a curve and $\iota:G\to \Aut(X)$ is an injective group homomorphism. We will suppress $\iota$ and simply say that $X$ is a $G$-curve.  Let $ Y=X/G$ (a smooth curve) and let $\pi:X\to Y=X/G$ be the covering map. We say that the ramification type of the $G$-curve $X$ is $\mathcal{C}=(C_1,\cdots, C_r)$  if  the ramified points on $Y$ can be labeled as $z_1,\cdots,z_r$ such that $C_i$ is  the conjugacy class in $G$ of the distinguished inertia group generator over $z_i$, i.e., the generator that acts on the tangent space at $z_i$ by $\exp(2\pi i/e_i)$ where $e_i$  is the ramification index, for  $1\leq i\leq r$.   Let $G_i, i=1,\dots, r,$ be the inertia subgroup of ramified points of the covering map  $\pi:X\to Y$, well-defined up to conjugacy so that $e_i=|G_i|$.  The signature of the $G$-curve $X$ is by definition the tuple  $(e_1,\cdots , e_r)$ which we arrange to be in increasing order.

Since our purpose is to find $X$ with $(H^1(X)^{\otimes 3})^G=0$, from now on we will assume that  $H^1(X)^G=0$, or equivalently the genus of $Y=X/G$ is   \begin{align}
\label{g(Y)}
g(Y)=0.  \end{align}
Assume that $g\geq2$. Let  
$\mathcal{H}(g,G,\mathcal{C})$ be the groupoid of $G$-curves $X$ of genus $g$ and ramification type $\mathcal{C}$ satisfying \eqref{g(Y)}, up to the obvious notion of isomorphism. Then $\mathcal{H}(g,G,\mathcal{C})$ has a structure of smooth Deligne--Mumford stack (over $\BC$) equipped with a universal family of $G$-curves. We denote this stack by the same notation and will call it the Hurwitz stack (for the data $(g,G,\mathcal{C})$) (\cite[\S6.2]{BR}; note that in {\it loc. cit.} the Hurwitz data over $\BC$ is equivalent to the ramification type here). 
Under the assumption $g(Y)=0$ \eqref{g(Y)},   the Hurwitz  stack $\mathcal{H}(g,G,\mathcal{C})$  is non-empty if and only if there exist $\gamma_i\in C_i$  for  $1\leq i\leq r$  such that 
\begin{align}
\label{exist}
\prod_{i=1}^r\gamma_i=1.
\end{align}
The Hurwitz  stack is equipped with a finite morphism $\mathcal{H}(g,G,\mathcal{C})\to\mathcal{M}_g$ and whenever it is non-empty, every irreducible component has dimension $r-3$.

For any $G$-curve $X$, the induced representation of $G$ on $H^1(X)$  is determined by the  ramification type $\mathcal{C}$ via the following formula in the Grothendieck group of $G$-representations over $\BC$ (note that $g(Y)=0$ by our assumption \eqref{g(Y)})
   \begin{align}\label{H1 repn}
H^1(X)=  -2(\Ind_{\{1\}}^G(1)-\Ind_G^G(1))+\sum_{i=1}^r (\Ind_{\{1\}}^G(1)-\Ind_{G_i}^G(1) ),
  \end{align}
  where every $(1)$ denotes the  trivial character of the corresponding group. Note that the right hand side depends only on the ramification type $\mathcal{C}=(C_1,\cdots, C_r)$. 
  
  The formula is implicitly in \cite[\S3]{Mac2} and can be proved using Lefschetz fixed point formula. 
    In particular, evaluating the trace of \eqref{H1 repn} at the identity of $G$ gives the Riemann--Hurwitz formula 
   \begin{align}\label{Hurwitz}
\frac{(2g-2)}{|G|}=-2+\sum_{i=1}^r(1-\frac{1}{e_i}).
  \end{align}
 Note that the equation is equivalent to  the assumption $g(Y)=0$ \eqref{g(Y)}.

 We recall the following result on the automorphism group of a hyperelliptic curve, pointed to us by B. Totaro.  
\begin{lem}\label{lem:hyp}Let $X$ be a hyperelliptic curve  of genus $\geq 2$. Then $\Aut(X)$
is a central extension of a finite subgroup of $\PGL_2(\BC)$ by $\BZ/2\BZ$ (generated by the unique hyperelliptic involution). Moreover, the order of $\Aut(X)$ is at most $\max\{8(g+1),120\}$. 
\end{lem}
\begin{proof}The first part is well-known due to the uniqueness of  hyperelliptic involution $\sigma$ and the fact that the quotient $\Aut(X)/\pair{\sigma}$ acts faithfully on the quotient $X/\pair{\sigma}\simeq \BP^1$.
A complete list of $\Aut(X)$ is obtained in \cite[Theorem 2.1]{BGG}, from which we deduce that the either $|\Aut(X)|\leq 8(g+1)$ (case (3.b) of Table 2 in {\it loc. cit.}) or  $|\Aut(X)|\leq 120$.
\end{proof}

\begin{cor}
\label{cor:Hur}
Hurwitz curves are never hyperelliptic.
\end{cor}
\begin{proof}
For a Hurwitz curve $X$,  its genus $g\geq 3$ and the order of $\Aut(X)$ is $84(g-1)$. Therefore we have $|\Aut(X)|\geq \max\{8(g+1),120\}$. By   Lemma \ref{lem:hyp} the proof is complete.
\end{proof}

 In \S\ref{s:ex},
we will use Lemma \ref{lem:hyp} repeatedly when we need to show that certain curves are non-hyperelliptic.

\subsection{Trilinear forms for $\PGL_2(\BF_q)$}Let $G=\PGL_2(\BF_q)$,  $U$ the subgroup of unipotent upper triangular matrices, $S$ the subgroup of diagonal matrices, $T$ a non-split maximal torus, and $B=US=SU$.
 We consider the generic (i.e., the restriction to $U$ contains one hence every nontrivial character of $U$) irreducible representations of $G$. Let $S'$ (resp. $T'$) be the normalizer of $S$ (resp. $T$) in $G$. Let $\pi^{S'}$ (resp. $\pi^{T'}$) be the unique representation of $S'$ (resp. $T'$) 
such that the restriction  $\pi|_{S}-\pi^{S'}|_S$ (resp.  $\pi|_{T}+\pi^{T'}|_T$, note the sign change) is the regular representation of $S$ (resp. $T$). More explicitly,
there are three cases:
\begin{altenumerate}
\item \label{Pr} if
 $\pi$ is a  principal series $\Ind(\eta):=\Ind_B^G(\eta)$ for a character $\eta$ of $S$ ($\eta\neq\eta^{-1}$), then $\pi^{S'}=\Ind_{S}^{S'}(\eta)$ and  $\pi^{T'}=0$.
 \item \label{St} if  $\pi={\rm St}\otimes\eta\circ\det$ is a twist of the Steinberg representation ${\rm St}$ ($\eta$ is a quadratic character of $\BF_q^\times$), then $\pi^{S'}=\eta\circ \det|_{S'} $ and $\pi^{T'}=\eta\circ \det|_{T'}$.
 \item \label{Cusp} if $\pi$ is a cuspidal representation $\pi_\chi$ attached to a character $\chi$ of $T$ ($\chi\neq\chi^{-1}$), then $\pi^{S'}=0$ and  $\pi^{T'}=\Ind_T^{T'}(\chi)$.
 \end{altenumerate}

 We have an analog over finite field of Prasad's theorem on trilinear forms over local fields. A special case has appeared in \cite[Example 2.5]{Pra2}.
 For $H\in\{G,S',T'\}$, we define 
 $$
 m_H(\pi_1,\pi_2,\pi_3)=\dim (\pi_1^H\otimes\pi_2^H\otimes\pi_3^H )^H,
 $$
 where $\pi_i^G$ is understood as $\pi_i$.
 
\begin{thm}\label{thm:trilinear}
Let $\pi_1,\pi_2,\pi_3$ be irreducible generic representations of $\PGL_2(\BF_q)$.
Then 
$$
\sum_{H\in\{G,S',T'\}} \sgn(H) m_H(\pi_1,\pi_2,\pi_3)=1,
$$
where $\sgn(G)=\sgn(T')=1$ and $\sgn(S')=-1$.
\end{thm}

     \begin{proof}

    Let $f$ be a conjugate invariant function on $G=\PGL_2(\BF_q)$. Then we have the following integration formula by summing over all conjugacy classes 
\begin{align} \label{eq:f G}
   \frac{1}{|G|}\sum_{g\in G} f(g)=\frac{1}{|G|}f(1)+\frac{1}{|U|}f(u)+\frac{1}{2|S|}\sum_{s\in S, s\neq 1}f(s)+\frac{1}{2|T|}\sum_{t\in T, t\neq 1}f(t).
  \end{align} 
  We will apply the formula to the product $f=\chi_{\pi_1}\chi_{\pi_2}\chi_{\pi_3}$ of characters.  We use  the character table of $G$,  where $u$ denotes a fixed non-identity element in $U$:
\begin{center}
\begin{tabular}{|c|c|c|c|c|c|c|}\hline
conjugacy classes  &  $1$    &$u$    & $s\neq 1$    &$t\neq 1$    \\
\hline
principal series   $\Ind(\eta)$   & $ q+1 $     &$1 $   &  $\eta(s)+\eta^{-1}(s)$  &$0$   \\
\hline
Steinberg twist   ${\rm St}\otimes\eta$  &  $ q $     &$0 $   &  $\eta\circ\det(s)$  &$-\eta\circ\det(t)$  \\
\hline
cuspidal  $\pi_\chi$  &  $ q-1 $     &$-1$   &  $0$  &$-\chi(t)-\chi^{-1}(t)$ \\
\hline
\end{tabular}
\end{center}

We claim
\begin{align} \label{eq:claim}
\sum_{H\in\{G,S',T'\}}& \sgn(H) m_H(\pi_1,\pi_2,\pi_3)=\\
&\frac{1}{|G|}\chi_{\pi_1}\chi_{\pi_2}\chi_{\pi_3}(1)+\frac{1}{|U|} \chi_{\pi_1}\chi_{\pi_2}\chi_{\pi_3}(u) -\frac{1}{2|S|}\chi_{\pi^{S'}_1}\chi_{\pi^{S'}_2}\chi_{\pi^{S'}_3} (1) +\frac{1}{2|T|}  \chi_{\pi^{T'}_1}\chi_{\pi^{T'}_2}\chi_{\pi^{T'}_3} (1)\notag.
\end{align}
 
   In fact, if at least one of the three representations, say $\pi_i$, is not a twist of Steinberg, then by \eqref{Pr} and \eqref{Cusp}, 
     we have  (note that $S$ is a normal subgroup of $S'$)
\begin{align} \label{eq:char S}
   \chi_{\pi^{S'}_i}(s)=\begin{cases}   \chi_{\pi_i}(s), &\text{if } s\in S\setminus\{1\},\\
   0,&\text{if $s\in S'$ is not (conjugate to any element) in $S$.}
   \end{cases}
\end{align} 
Then we see that 
    \begin{align*}
    m_{S'}(\pi_1,\pi_2,\pi_3)&=\frac{1}{2|S|}\sum_{s\in S} \chi_{\pi^{S'}_1}\chi_{\pi^{S'}_2}\chi_{\pi^{S'}_3}(s)\\
    &=\frac{1}{2|S|} \chi_{\pi^{S'}_1}\chi_{\pi^{S'}_2}\chi_{\pi^{S'}_3}(1)+\frac{1}{2|S|}\sum_{s\in S\setminus\{1\}}\chi_{\pi_1}\chi_{\pi_2}\chi_{\pi_3}(s).
    \end{align*}
    (For the first equality, the contribution from $s\in S'\setminus S$ vanishes by the character formula \eqref{eq:char S}.)     
There is a similar identity for $\chi_{\pi^{T'}_i}$ and the claim follows from \eqref{eq:f G} (for $f=\chi_{\pi_1}\chi_{\pi_2}\chi_{\pi_3}$).  

  If all  of the $\pi_i$'s are  twists of Steinberg, say $\pi_i={\rm St}\otimes\eta_i\circ\det$, we set $\eta=\prod_{i=1}^3 \eta_i$. By \eqref{St} and the fact that $\det|_{S}$ is surjective onto $\BF_q^\times$, we obtain
    \begin{align*}
  \frac{1}{2|S|}\sum_{s\in S'\setminus S} \chi_{\pi^{S'}_1}\chi_{\pi^{S'}_2}\chi_{\pi^{S'}_3}(s)&=\frac{1}{2|S|}\sum_{s\in S}\eta\circ\det(s)\\
  &=\begin{cases} 0,& \text{if $\eta$ is nontrivial},\\
1/2, &  \text{if $\eta$ is trivial}.
  \end{cases}
    \end{align*}
  We have a similar identity for $T$ replacing $S$ and hence we obtain$$
  \frac{1}{2|S|}\sum_{s\in S'\setminus S} \chi_{\pi^{S'}_1}\chi_{\pi^{S'}_2}\chi_{\pi^{S'}_3}(s)=  \frac{1}{2|T|}\sum_{s\in T'\setminus T} \chi_{\pi^{T'}_1}\chi_{\pi^{T'}_2}\chi_{\pi^{T'}_3}(t)\in\{0,1/2\}.
  $$
    Then the same argument shows the claimed equality \eqref{eq:claim}.
  
Next, to apply \eqref{eq:claim},  we note that  (e.g., from the character table)
     $$
    \begin{cases}
     \chi_{\pi}(1)=\chi_{\pi}(u)+q,
     \\\chi_{\pi^{S'}}(1)=\chi_{\pi}(u)+1,\\
     \chi_{\pi^{T'}}(1)=-\chi_{\pi}(u)+1.
    \end{cases} $$
     Denote $x_i= \chi_{\pi_i}(u)$. It suffices to show
     $$
     \frac{(q+x_1)(q+x_2)(q+x_3)}{(q-1)q(q+1)}+\frac{x_1x_2x_3}{q}-\frac{(1+x_1)(1+x_2)(1+x_3)}{2(q-1)}+\frac{(1-x_1)(1-x_2)(1-x_3)}{2(q+1)}=1.
     $$
     It is straightforward to check that the above identity holds (in fact it holds as a rational function in the four variables $q,x_1,x_2,x_3$). The proof is complete.  
      \end{proof}

\begin{rmk}
If any of $\pi_i$ is cuspidal, then $m_{S'}(\pi_1,\pi_2,\pi_3)=0$ and we obtain
$$
\dim\, (\pi_1\otimes \pi_2\otimes \pi_3)^G+\dim\, (\pi^{T'}_1\otimes \pi^{T'}_2\otimes \pi^{T'}_3)^{T'}=1.
$$
If any of $\pi_i$ is a principal series, then $m_{T'}(\pi_1,\pi_2,\pi_3)=0$ and we obtain
$$
\dim\, (\pi_1\otimes \pi_2\otimes \pi_3)^G-\dim\, (\pi^{S'}_1\otimes \pi^{S'}_2\otimes \pi^{S'}_3)^{S'}=1.
$$
 If all of $\pi_i$ are twists of Steinberg representation, then by \eqref{St} it is easy to see that $m_{S'}(\pi_1,\pi_2,\pi_3)=m_{T'}(\pi_1,\pi_2,\pi_3)\in\{0,1\}$ and we obtain 
 
 $$
\dim\, (\pi_1\otimes \pi_2\otimes \pi_3)^G=1.
$$
These equalities allow us to compute $\dim\, (\pi_1\otimes \pi_2\otimes \pi_3)^G$ easily in all cases.

\end{rmk}

\begin{cor}\label{cor:tri}
Let $\pi$ be a generic irreducible representation of $G=\PGL_2(\BF_q)$. Then 
$\dim\, (\pi^{\otimes 3})^G=1$ unless $\pi$ is a cuspidal representation attached to a nontrivial cubic character of $T$, in which case $\dim\, (\pi^{\otimes 3})^G=0$.
\end{cor}     

This result will be applied twice later: the Fricke--Macbeath curve in \S\ref{Fricke--Macbeath curve1}  and the Bring curve in \S\ref{ss:Bring}.

\begin{rmk}
We indicate the relation to Prasad's trilinear form for representations of $\PGL_2$ over local fields. Let $F$ be a non-archimedean local field with ring of integers $O_F$ and residue field $k=\BF_q$. For irreducible cuspidal representations $\pi_i, 1\leq i\leq 3$ of $\PGL_2(k)$, by inflation we obtain representations $\wt\pi_i$ of $\PGL_2(O_F)$.
Then  the compact induction  $\Pi_i=c\,\Ind^{\PGL_2(F)}_{\PGL_2(O_F)}\wt\pi_i$  of $\wt\pi_i$ is a supercuspidal representation of $\PGL_2(F)$. Then 
using the facts after \cite[Prop. 6.7]{Pra} one can show
$$
\dim \Hom_{\PGL_2(F)}(\Pi_1\otimes\Pi_2\otimes\Pi_3,\BC)=\dim\Hom_{\PGL_2(k)}(\pi_1\otimes\pi_2\otimes\pi_3,\BC).
$$
Let $D$ be the division quaternion algebra over $F$. Then we have $D^\times/F^\times O_D^\times\simeq T'$ and we may view $\pi_i^{T'}$ as a representation of ${\rm P}D^\times=D^\times/F^\times$, which is exactly the Jacquet--Langlands correspondence  $\Pi_i'$ of $\Pi_i$.   Then our result translates to Prasad's theorem \cite{Pra} that $$
\dim \Hom_{\PGL_2(F)}(\Pi_1\otimes\Pi_2\otimes\Pi_3,\BC)+\dim \Hom_{{\rm P}D^\times}(\Pi'_1\otimes\Pi'_2\otimes\Pi'_3,\BC)=1.
$$

\end{rmk}

  \subsection{Fricke--Macbeath curve and its genus 3 quotient}  \label{Fricke--Macbeath curve1}

Let $X$ be the unique (up to isomorphism) Hurwitz curve of genus 7 over $\BC$, known as the Fricke--Macbeath curve \cite{Mac1}. 
 Its automorphism group
is isomorphic to  $\PGL_2({\BF_8})\simeq \PSL_2({\BF_8})$.

   \begin{thm}[Bisogno--Li--Litt--Srinivasan    \cite{BLLS},  Gross \cite{Gro}]\label{Gro1}  
   There are  no   nonzero diagonal-$\PGL_2({\BF_8})$-invariant elements in
   $H^1(X)^{\otimes 3}$.
\end{thm}  

 Let $\pi_{\chi_3}$ denote  the cuspidal representation of $\PGL_2(\BF_8)$ attached to the unique order-$3$ character $\chi_3$ (up to taking its inverse) of the anisotropic torus $T\simeq \BF_{8^2}^\times/ \BF_{8}^\times\simeq\BZ/9\BZ$. We sketch a ``pure thought" proof that $$H^1(X)=2\pi_{\chi_3},$$ which implies the desired assertion $(H^1(X)^{\otimes 3})^G=0$ by Corollary \ref{cor:tri}.

For  $i=2,3,7$, let $H_i$ be a cyclic subgroup of order $i$, which is unique up to conjugacy.  We may assume that $H_2$ is generated by any nontrivial element of order 2 in the unipotent radical $U\simeq\BF_8$ of the Borel subgroup, $H_3$ is the 3-torsion subgroup of $T$, and $H_7$ is $S$. Then, by \eqref{H1 repn} and Frobenius reciprocity, for any non-trivial irreducible representation $\pi$, its multiplicity in $H^1(X)$  is equal to 
     $$
     m(\pi)=\dim \pi-\dim \pi^{H_2}-\dim \pi^{H_3}-\dim \pi^{H_7}.
     $$
There are four $7$-dimensional cuspidal representations called  $\pi_{\chi}=\pi_{\chi^{-1}}$, attached to $\chi:T\to\BC^\times$ such that $\chi\neq \chi^{-1}$.  Recall (see for example \cite[\S7]{Gro2}) that the restriction of $\pi_\chi$ to the  torus $S$ is the regular representation, the restriction to $T$ is the regular representation minus $\chi,\chi^{-1}$, and the restriction to the unipotent radical $U$ is the  regular representation minus the trivial representation. It follows that all four cuspidal representations  the property: $\dim \pi^{H_7}=1,\dim \pi^{H_2}=4-1=3$. Moreover, we have $\dim \pi_\chi^{H_3}=3-2$ (resp. $3-0$) if $\chi|_{H_3}=1$  (resp. $\chi|_{H_3}\neq 1$), namely $\chi$ is of order $3$ (resp. of order $9$). Hence $m(\pi)=7-3-1-1=2$ (resp. $7-3-3-1=0$) if $\chi$ is of order $3$ (resp. of order $9$). By dimension reason, we must have $H^1(X)=2\pi_{\chi_3}$.

 Now by Theorem \ref {good}, we have $[\Delta_{\xi_X}]=0$. 
 Let $Y$ be the quotient curve of the Fricke--Macbeath curve $X$ under an involution (all order $2$ elements are conjugate). 
 \begin{cor}\label{cor:qt FM}
  The curve $Y$ is a genus $3$ non-hyperelliptic curve such that  $\pi_\ast(\xi_X)=\xi_Y\in \Ch^1(Y)$ and $[\Delta_{\xi_Y}]=0 \in \Ch^2(Y^3)$.
  \end{cor}
  
  \begin{proof}
 By \cite[\S 4]{TV}, $Y$ is non-hyperelliptic.   The rest follows from Proposition \ref{prop:pf}.
\end{proof}
In particular, $X$ is non-hyperelliptic, which can  also be checked directly since we know $\Aut(X)$.
 \begin{lem}\label{lem:qt FM}
  We have $(H^1(Y)^{\otimes 3})^{\Aut(Y)}\neq 0$.

     \end{lem}
   \begin{proof}
By a direct computation using    \cite[\S 6, Table 1, Table 2]{MSSV}, 
we find that the  only genus $3$ non-hyperelliptic curve  $C$ with $(H^1(C)^{\otimes 3})^{\Aut(C)}= 0$  has plane model 
        $y^4=x^3-1$ (see \ref{y4x3-1}).  By computing Dixmier--Ohno  invariants \cite{Dix,Ohn}, we found that  
        $Y$ is not isomorphic to $C$ over $\BC$. The lemma follows.
        \end{proof}

   \subsection{Hurwitz curves with automorphism groups $G=\PSL_2(\BF_q)$} 

We  consider all Hurwitz curves with automorphism group $G=\PSL_2(\BF_q)$ with $q=p^m$ for a rational prime $p$. By a theorem of Macbeath, the necessary and sufficient condition for  $\PSL_2(\BF_q)$ to arise as the automorphism group of a Hurwitz curve is that $q=7$, or $q=p$ with $p\equiv\pm 1\, \mod 7$, or $q=p^3$ with with $p\equiv\pm  2, \text{or } \pm 3\, \mod 7$ \cite[\S3]{Con}. 

\begin{thm}\label{thm:Hur PSL}
The  Fricke--Macbeath curve is the unique Hurwitz curve $X$ with $\Aut(X)\simeq\PSL_2(\BF_q)$ that has the property $(H^1(X)^{\otimes 3})^{\Aut(X)}=0$.
\end{thm}
First of all, it is easy to check the assertion case-by-case when $q<43$. In general, we note that  the group $\PSL_2(\BF_q)$ has the property that there is at most one conjugacy class of cyclic subgroups of a given order. Note that  a Hurwitz curve $X$ must be a $G=\Aut(X)$-covering of $\BP^1\simeq X/G$ of signature $(2,3,7)$ by the Riemann--Hurwitz formula \eqref{Hurwitz}. We let $G_i$ be a cyclic subgroup of $G$ of order $e_i$  (unique to to conjugacy) for  $e_1=2,e_2=3,e_3=7$. We let ${\rm Hur}$ denote the representation defined by  \eqref{H1 repn}. It may or may not arise from the Hurwitz representation for some $G$-curve $X$ (we have not imposed the condition \eqref{exist}). Every Hurwitz representation arising from a Hurwitz curve with automorphism group $G=\PSL_2(\BF_q)$ must be of such form. We prove the following  assertion which implies Theorem \ref{thm:Hur PSL}.
 \begin{prop}When $q\geq 43$ is odd,
the representation ${\rm Hur}$ of $G=\PSL_2(\BF_q)$ has the property $({\rm Hur}^{\otimes 3})^{G}\neq 0$.
\end{prop}

\begin{proof}
We list the following properties of the simple group $\PSL_2(\BF_q)$ \cite[\S4]{Mac2}. (See \cite[\S2]{Gro2} for more details when $q=p>3$ is an odd prime.)
\begin{itemize}
\item 
The order of the group is
$| \PSL_2(\BF_q)|=
\frac{1}{2}q(q+1)(q-1).$
\item 

For every divisor $d>2$ of $(q-1)/2$ (resp. $(q + 1)/2$), there are $\phi(d)/2$ conjugacy classes of order $d$ in $G$ with size $2|G|/(q-1)$ (resp. $2|G|/(q+1)$). 
\item When $d = 2$, there is a unique conjugacy class of order $d$ with size $|G|/(q-1)$ (resp. $|G|/(q+1)$) if $2|(q-1)/2$ (resp. $2|(q+1)/2$).
\item There are two unipotent conjugacy classes of order $p$ with size $|G|/q$. \item    Each element is conjugate to  one in either the split torus $S$, the anisotropic torus $T$, or a cyclic subgroup $\simeq \BF_p$ of the unipotent $U$. 
 \end{itemize}
 
Let $g\in \PSL_2(\BF_q)$ be an element of order $d\neq 1$. Then  by \eqref{H1 repn} (also computed explicitly in \cite[\S4]{Mac2}) we have
$$
\chi(g):=\Tr(g|{\rm Hur})=2-\begin{cases}
(q-1)\sum_{d|e_i}\frac{1}{e_i},& \text{ if } d|\frac{q-1}{2} \\
(q+1)\sum_{d|e_i}\frac{1}{e_i}, &\text{ if } d|\frac{q+1}{2}\\
\frac{(n,2)}{2}q(1-p^{-1})\sum_{e_i=p}1,&\text{ if } d=p.
\end{cases}
$$ 
It follows that $\chi(g)=2$ unless $g=1$ or $g$ is conjugate to an element in $G_i$.  We have $\chi(1)=2+\frac{|G|}{42}>\frac{|G|}{42}$. The number of the conjugacy class of order $i=2,3,7$ is $1,1,3$ respectively and the size is at most $ \frac{2|G|}{q-1}$, and the value  of $\chi$ at such a conjugacy class is bounded below by $-q$.  By omitting the other conjugacy classes, we obtain
\begin{align*}
\sum_{g\in G}\chi(g)^3&> \left(\frac{|G|}{42}\right)^3- \frac{2|G|}{q-1}q^3(1+1+3)\\
&> \frac{(q-1)^3q^3(q+1)^3}{84^3}-   \frac{(q-1)q(q+1)}{q-1}q^3(1+1+3)
\\&>\frac{q^4(q+1)}{84^3}((q-1)^3 (q+1)-5\cdot 2^3\cdot 42^3). \end{align*}
If  $q\geq 43$ we have $(q-1)^3 (q+1)-5\cdot 2^3\cdot 42^3\geq 42^3\cdot (44-5\cdot 2^3)>0$. 
Therefore $\sum_{g\in G}\chi(g)^3>0$ and the  space of $G$-invariants $({\rm Hur}^{\otimes 3})^{G}$ must be non-zero.
\end{proof}

 \section{More examples}\label{s:ex}
 The purpose of this section is to show that our vanishing Theorem \ref{good}  is applicable to a rather large number of examples in  genus $3\leq g\leq 5$  listed in  \cite{MSSV}, by    
  exhibiting  non-hyperelliptic curves there with the property $(H^1(X)^{\otimes 3})^G=0$.
 We give ``pure thought" proofs of this equality whenever possible (e.g., Bring's curve in \S\ref{ss:Bring} and the 1-dimensional family in \S\ref{ss:1-dim}). The proofs for the remaining cases use character tables together with straightforward, albeit tedious, computation, which we feel are more suitable to appear in a separate document \cite{QZA}.
 Using Lemma \ref{lem:hyp} and   Klein's classification of finite subgroups of $\PGL_2(\BC)$, it is easy to see that none of the curves below is hyperelliptic.
    \subsection{0-dimensional examples} 
   The following examples arise from zero-dimensional Hurwitz stacks $\mathcal{H}(g,G,\mathcal{C})$. (In fact, by \cite[\S6, 7.1]{MSSV}, every example below has $G$ as the full automorphism group.)  
        \subsubsection{$g=3$: $y^4=x^3-1$}\label{y4x3-1}
               By \cite[Table 2]{MSSV}, the    curve $X$ with plane model 
        $y^4=x^3-1$  has $g=3$  and
    is a  $G$-curve of signature $(2,3,12)$, where $G$ is   the non-split central extension by $\BZ/4\BZ$ of $A_4$, which 
has Group ID  (48,33)  in the Small Groups Library.
 In fact, this is only genus $3$ non-hyperelliptic curve  $X$ with $(H^1(X)^{\otimes 3})^{\Aut(X)}= 0$, as \cite[\S 6]{MSSV}  exhausts all curves of genus 3 (while for higher genera, \cite[\S 6]{MSSV}   only includes curves with large automorphism groups). 
     
           \subsubsection{ $g=4$: Bring's curve}\label{ss:Bring}
           Among all  genus $4$ curves, the Bring curve $X$  is the unique one with the largest possible automorphism group:
            $$\Aut(X)=G:=S_5\simeq \PGL_2(\BF_5),$$ which has Group ID  (120,34)  from the Small Groups Library. It is a  $G$-curve of signature $(2,4,5)$. It admits an explicit model
     $$
\left\{     \sum_{i=1}^5 X_i=     \sum_{i=1}^5 X_i^2=    \sum_{i=1}^5 X_i^3=0\right\}\subset \BP^4
     $$      
     on which $S_5$ acts by permuting the coordinates $X_i$.
  By the same argument after Theorem \ref{Gro1}, we can show
     $$
     H^1(X)=2\pi_{\chi_3}
     $$where 
    $\pi_{\chi_3}$ is  the cuspidal representation of $\PGL_2(\BF_5)$ attached to the unique order-$3$ character (up to taking its inverse) of the anisotropic torus $T$. By Corollary \ref{cor:tri}, we have   $(H^1(X)^{\otimes 3})^G=0$.

   \subsubsection{  $g=4$: $y^3 = (x^3-1)^2(x^3 + 1) $.}  \label{7240} 
  
By \cite[Table 4]{MSSV}, there is a unique   $G$-curve $X$ of $g=4$
    with $G$    the wreath product
 $ S_3\wr \BZ/2\BZ
   $, which has Group ID  (72,40)  from the Small Groups Library. It has signature (2,4,6). 
   It is cyclic trigonal  and its equation is found in \cite{Sw}. 
    
     \subsubsection{ $g=5$: Wiman's curve} 
     By \cite[Table 4]{MSSV}, there is a unique   $G$-curve $X$ of $g=5$
       with       $G= (\BZ/2\BZ)^4\rtimes D_5
     $, which has 
     Group ID  $(160,234)$  from the Small Groups Library.   It has
   signature $(2,4,5)$.  It was first studied by Wiman as we learnt from \cite{Sw}. Moreover,   its equation can also be found in \cite{Sw}.

     \subsubsection{  $g=5$: another Wiman's curve}  
     By \cite[Table 4]{MSSV}, there is a unique   $G$-curve $X$ of $g=5$
       with     
 $ G =\GL_2(\BZ/4\BZ)$, which has 
     Group ID  (96,195) from the Small Groups Library.   It has
 signature $(2,4,6)$. It was  also studied by Wiman  and    its equation can also be found in \cite{Sw}.

      \subsection{1-dimensional families}
        The following examples arise from 1-dimensional Hurwitz stacks $\mathcal{H}(g,G,\mathcal{C})$.  (In fact, from \cite[7.1]{MSSV},  it is not hard to  deduce that the generic fiber of each example has   $G$ as the full automorphism group.)
        Their equations   could be computed as \cite{Sw} (though not done here or in \cite{Sw}). In particular, it is not hard to show that they are in fact defined over number fields. 
      \subsubsection{$g=4$}\label{ss:1-dim}
      By \cite[Table 4]{MSSV}, there is a unique 1-dimensional family of $G$-curves of $g=4$ and signature $(2,2,2,3)$
   where $G=S_3\times S_3$, which has  Group ID   (36,10) from the Small Groups Library. (The existence can be shown by the theory of Hurwitz stack recalled in \S\ref{ss:Hur}; see the next paragraph.)

  This is the universal family over the Hurwitz stack $\mathcal{H}(4,G,\mathcal{C})$ for the choice of  $\mathcal{C}$ given below. We denote $
H=S_3
$
and let $H_i$, $i=2,3$, be a subgroup of $H$ of order $i$, which is unique up to conjugacy.  The four cyclic subgroups of $G=H\times H$ appearing in the Hurwitz representation \eqref{H1 repn}   are (up to conjugacy) 
$$
H_2\times\{1\}, \{1\}\times H_2, \Delta(H_2), \Delta(H_3).
 $$ 
It is easy to see that one  can arrange a choice of the generator of a suitable conjugacy of each of the four subgroups such that their product is $1$ and they generate the group $G$.  Choosing $\mathcal{C}$ accordingly and then by \eqref{exist}, the Hurwitz stack $\mathcal{H}(4,G,\mathcal{C})$  is non-empty and has dimension $r-3=4-3=1$.

Now we compute the induced representations appearing in \eqref{H1 repn}, using  $\Hom_{G}(\Ind_{H}^G(1),\pi)\simeq \pi^H$ (Frobenius reciprocity) and noting that  the regular representation of $H$ decomposes:  
   $$
   \Ind_ {\{1\}}^{S_3}(1)=(1)+\sgn+2\rho,
   $$ 
   where $\sgn$ denotes the sign character and $\rho$ is the unique irreducible representation of dimension $2$.
  We enumerate the results:
   \begin{align*}
   \Ind_{H_2\times\{1\}}^{H\times H}(1)&=\Ind_{H_2}^H(1)\boxtimes \Ind_{\{1\}}^H(1)=((1)+\rho)\boxtimes ((1)+\sgn+2\rho)\\
   \Ind_{\{1\}\times H_2}^{H\times H}(1)&=\Ind_{\{1\}}^H(1) \boxtimes \Ind_{H_2}^H(1)= ((1)+\sgn+2\rho)\boxtimes ((1)+\rho)\\
    \Ind_{\Delta(H_2)}^{H\times H}(1)(1)&=(1)+\sgn\boxtimes\sgn+\sgn\boxtimes\rho+\rho\boxtimes\sgn+(1)\boxtimes\rho+\rho\boxtimes(1)+2\rho\boxtimes\rho\\ 
        \Ind_{\Delta(H_3)}^{H\times H}(1)&=(1)+(1)\boxtimes\sgn+\sgn\boxtimes(1)+\sgn\boxtimes\sgn+2\rho\boxtimes\rho.
\end{align*}
   It follows from \eqref{H1 repn} that for any curve $X$ in this family, 
   $$
   H^1(X)=2( \sgn\boxtimes \rho+\rho\boxtimes \sgn).
   $$
  By the lemma below we have $( H^1(X)^{\otimes 3})^G=0$.
   \begin{lem}We have
   $( (\sgn\boxtimes \rho+\rho\boxtimes \sgn)^{\otimes 3})^{H\times H}=0$.
   \end{lem}
   \begin{proof}
By looking at $H_2\times \{1\}$-action and by $(\sgn^{\otimes 3})^{H_2}=0$,  we have 
   $$
(   (\sgn\boxtimes \rho)^{\otimes 3})^{H\times H}\simeq (\sgn^{\otimes 3})^H\boxtimes( \rho^{\otimes 3})^H=0.
$$
By symmetry we have $(   (\rho\boxtimes \sgn)^{\otimes 3})^{H\times H}=0$. Next, since 
$$
(\sgn\boxtimes \rho)^{\otimes 2}\otimes (\rho\boxtimes \sgn)=\rho\boxtimes ( \rho^{\otimes 2}\otimes\sgn)
$$
and   $\rho^{H_3}=0$, by looking at $H_3\times \{1\}$-action,  we conclude that the left hand side has no $H\times H$-invariant. By symmetry we conclude that $(\sgn\boxtimes \rho)\otimes (\rho\boxtimes \sgn)^{\otimes 2}$ has no  $H\times H$-invariant. We have thus shown $( (\sgn\boxtimes \rho+\rho\boxtimes \sgn)^{\otimes 3})^{H\times H}=0$.
\end{proof}
 
      \subsubsection{$g=5$}
      By \cite[Table 4]{MSSV}, there is a unique 1-dimensional family of $G$-curves of $g=5$  where $G=(\BZ/2\BZ)^2\wr \BZ/2\BZ $, which has  Group ID    $(32,27)$  from the Small Groups Library.  It has  signature $(2,2,2,4)$.

\end{document}